\documentclass[12pt,twoside]{amsart}

\usepackage{amsmath, amssymb, amscd, paralist,tabularx,supertabular,amsmath, verbatim, amsthm, amssymb, mathrsfs, manfnt,times,latexsym,amscd,graphicx, color}
\usepackage{amsmath, amsthm, amssymb, eqnarray}
\usepackage{paralist}  
\usepackage{manfnt}    
\usepackage{url}
\usepackage[all]{xy}

\numberwithin{equation}{section}

\usepackage{fullpage}

\DeclareMathAlphabet{\curly}{U}{rsfs}{m}{n}

\newtheorem{thm}{Theorem}[section]
\newtheorem{theorem}[thm]{Theorem}
\newtheorem{corollary}[thm]{Corollary}
	
\newtheorem{proposition}[thm]{Proposition}

\newtheorem*{thm*}{Theorem}

\DeclareMathOperator{\Ram}{Ram}
\DeclareMathOperator{\sys}{Sys}
\DeclareMathOperator{\vol}{Vol}
\DeclareMathOperator{\isom}{Isom}

\newcommand{\p}{\mathfrak p}
\newcommand{\q}{\mathfrak q}

\theoremstyle{definition}

\theoremstyle{remark}

\newtheorem*{lemma*}{Lemma}

\title{Systoles of Arithmetic Hyperbolic $2$- and $3$-Manifolds}
\author{Laurel Heck}
\address{Department of Mathematics \\ University of Washington \\ Box 354350, Seattle, WA 98195}
\email{lheck98@uw.edu}

\author{Benjamin Linowitz}
\address{Department of Mathematics\\ 
10 North Professor Street\\
Oberlin, OH 44074}
\email{benjamin.linowitz@oberlin.edu}

\begin{document}

\begin{abstract} 
In this paper we study the systoles of arithmetic hyperbolic $2$- and $3$-manifolds. Our first result is the construction of infinitely many arithmetic hyperbolic $2$- and $3$-manifolds which are pairwise noncommensurable, all have the same systole, and whose volumes are explicitly bounded. Our second result fixes a positive number $x$ and gives an upper bound for the least volume of an arithmetic hyperbolic $2$- or $3$-manifold whose systole is greater than $x$. We conclude by determining, for certain small values of $x$, the least volume of a principal arithmetic hyperbolic $2$-manifold over $\mathbf Q$ or $3$-manifold over $\mathbf Q(i)$ whose systole is greater than $x$. 
\end{abstract}

\maketitle

\section{Introduction}

Let $M$ be a closed Riemannian manifold. The central problem of inverse spectral geometry is to determine to extent to which the geometry and topology of $M$ are determined by its Laplace eigenvalue spectrum. It is known, for example, that volume and scalar curvature are determined by a manifold's Laplace eigenvalue spectrum, whereas isometry class is not.  

In geometric topology one often studies not only the Laplace eigenvalue spectrum but the geodesic length spectrum as well. The geodesic length spectrum, defined as the multiset of lengths of closed geodesics, turns out to encode many of the same geometric and topological properties as the Laplace eigenvalue spectrum. In fact, the Selberg Trace Formula shows that if $M$ is a compact hyperbolic $2$-manifold, then these two spectra determine one another. Subsequent work of Duistermaat and Guillemin \cite{DG} and Duistermaat, Kolk, and Varadarajan \cite{DKV} shows that the geodesic length spectrum of a compact locally symmetric manifold of nonpositive curvature is determined by the manifold's Laplace eigenvalue spectrum.

When studying arithmetic manifolds one is often interested in knowing which geometric and topological properties determine the commensurability class of the manifold. (Two manifolds are said to be {\it commensurable} if they have an isometric finite-sheeted covering in common.) An important result of Reid \cite{R} states that two arithmetic hyperbolic $2$-manifolds with identical geodesic length spectra are necessarily commensurable. This result was extended to arithmetic hyperbolic $3$-manifolds by Chinburg, Hamilton, Long, and Reid \cite{CHLR}, and to a broad class of arithmetic locally symmetric spaces by Prasad and Rapinchuk \cite{PR}. 

A different generalization of the results of Reid and Chinburg, Hamilton, Long, and Reid was obtained by Linowitz, McReynolds, Pollack and Thompson \cite{LMPT}, who showed that  
two arithmetic hyperbolic $2$- or $3$-manifolds $M$ and $M'$ with volume less than $V$ are commensurable provided their length spectra begin with the same $f(V)$ lengths. Here $f(V)$ is an explicit (doubly exponential) function of the volume of the manifolds $M$ and $M'$. The dependence of the number of lengths needed to determine commensurability on volume is in fact necessary, as one can construct, for a very broad class of arithmetic hyperbolic $2$- or $3$-manifolds $M$, and any finite subset $S$ of the geodesic length spectrum of $M$, an arithmetic hyperbolic manifold $N$ which is not commensurable with $M$ and whose length spectrum also contains $S$ (see \cite[Section 4.4]{LMPT}).

In this paper we consider a variant of the problems discussed above. Recall that the {\it systole} of a compact Riemannian manifold $M$ is the length of the shortest non-contractible closed geodesic on $M$. In this paper we study the systolic geometry of arithmetic hyperbolic $2$- and $3$-manifolds. It is known, for example, that there exist non-commensurable hyperbolic manifolds with the same systole. For example, the work of Millichap \cite{Mi} and Futer-Millichap \cite{FMi} both contain constructions of pairs of non-commensurable hyperbolic $3$-manifolds whose length spectra agree for the first $n$ lengths. It follows that the manifolds constructed by Millichap and Futer-Millichap have the same systoles. Although the construction of Futer-Millichap produces pairs of manifolds, the construction of Millichap actually produces families of manifolds of arbitrarily large (though finite) cardinality.

Our first main result provides a strong generalization of this fact by exhibiting {\it infinite} families of non-commensurable arithmetic manifolds whose volumes  may be explicitly bounded and which all have the same systole.
 
\begin{thm*}
Let $M$ be a hyperbolic $2$-manifold (respectively $3$-manifold) whose fundamental group is principal arithmetic. Then there exist infinitely many pairwise non-commensurable hyperbolic $2$-manifolds (respectively $3$-manifolds) $M_1, M_2, M_3, \dots$ such that
\begin{enumerate}
\item all of the $M_i$ have the same systole as $M$, and
\item $\vol(M_{n})\leq c \cdot n^2  \cdot \vol(M_{n-1})$ for some constant $c$ depending only on $M$.
\end{enumerate}
\end{thm*}

One of the most well known conjectures concerning the systolic geometry of arithmetic manifolds is the Short Geodesics Conjecture, which asserts that there is a constant $C>0$ such that every arithmetic hyperbolic $2$- or $3$-manifold has systole greater than $C$. There is an analogous conjecture for more general arithmetic locally symmetric spaces (see \cite{G}). The Short Geodesic Conjecture is known to follow from Lehmer's Conjecture on Mahler measures, which asserts that there is a constant $\mu>1$ such that the Mahler measure of any nonzero algebraic number that is not a root of unity is greater than $\mu$. For a more thorough discussion of the Short Geodesic Conjecture, see Chapter 12 of Maclachlan and Reid \cite{MR}.

Although the Short Geodesic Conjecture is currently out of reach, it is easy to construct arithmetic manifolds with arbitrarily large systole. Although such constructions may be easily carried out using orders in quaternion algebras, estimating the resulting manifold's volume is another matter and is a much more nuanced problem. Our second main result achieves such an estimate.

\begin{thm*}
Fix a real number $x>0$. There exists a principal arithmetic hyperbolic $2$-manifold with systole at least $x$ and area at most \[4^{ce^{11x}}.\] Here $c$ is a positive, effectively computable constant. 
\end{thm*}

An analogous result is proven for arithmetic hyperbolic $3$-manifolds.

We expect that the bound given in the above theorem is far from optimal. That being said, obtaining a significant improvement seems to be a challenging problem. In Section \ref{subsectiongrh} we explain that our bound's doubly exponential character is closely related to the algebraic number theory problem of bounding the least prime which splits completely in a number field, and that even the Generalized Riemann Hypothesis would not allow us to significantly improve our bound. We conclude by using the SAGE computer algebra system to determine, for certain small values of $x$, the least volume of a principal arithmetic hyperbolic $2$-manifold over $\mathbf Q$ or $3$-manifold over $\mathbf Q(i)$ whose systole is greater than $x$.

\textbf{Acknowledgements}. The work of the first author is partially supported by an NSF Graduate Research Fellowship. The work of the second author is partially supported by NSF Grant Number DMS-1905437.
\section{Notation}

Unless explicitly stated otherwise, throughout this paper $k$ will denote a number field with ring of integers $\mathcal O_k$. We denote the absolute discriminant of $k$ by $\Delta_k$, and for a relative extension $L/k$ of number fields, by $\Delta_{L/k}$ the relative discriminant of the extension. Given an ideal $I$ of $\mathcal O_k$, we denote by $N(I)$ the norm of $I$. More generally, given a finite place $\nu$ of $k$, we denote by $N(\nu)$ the norm of the associated prime ideal. The set of prime ideals of $\mathcal O_k$ is denoted $\mathbb P_k$. Finally, we denote by $\zeta_k(s)$ the Dedekind zeta function of $k$.
\section{Preliminaries on quaternion algebras}

In this section we review some of the properties of quaternion algebras over number fields that will be used throughout this paper. A comprehensive reference for this material is Voight \cite{V}.

Given a field $k$ of characteristic not equal to $0$, we define a {\it quaternion algebra} $B$ over $k$ to be a four dimensional central simple $k$-algebra. Equivalently, $B$ is a quaternion algebra over $k$ if there exist non-zero elements $a,b\in k^\times$ such that $B$ is isomorphic to the central simple $k$-algebra with basis $\{1,i,j,ij\}$ and defining relations \[i^2=a,\qquad\qquad j^2=b,\qquad \qquad ij=-ji.\] It is a consequence of Wedderburn's Structure Theorem that if $B\not\cong \mathrm{M}_2(k)$, then $B$ is a division algebra.

Let $k$ be a number field and $B$ be a quaternion algebra over $k$. A place $\nu$ of $k$ is said to {\it ramify} in $B$ if the $k_\nu$-algebra $B\otimes_k k_\nu$ is a division algebra, and is {\it unramified} otherwise. Here $k_\nu$ denotes the completion of $k$ at the place $\nu$. We denote by $\Ram(B)$ the set of places of $k$ that ramify in $B$, and by $\Ram_f(B)$ the set of finite places of $k$ that ramify in $B$. It is known that the set $\Ram(B)$ is finite, of even cardinality, and that this set both determines and is determined by the isomorphism class of $B$ as a central simple $k$-algebra. As we have already noted, $\Ram(B)$ is empty if and only if $B\cong \mathrm{M}_2(k)$, and is a division algebra otherwise.

A {\it quaternion order} is a subring $\mathcal O$ of $B$ which is also a finitely generated $\mathcal O_k$-module containing a basis for $B$ over $k$. An order is said to be {\it maximal} if it is maximal with respect to inclusion. For example, the order $\mathrm{M}_2(\mathcal O_k)$ is always a maximal order of $\mathrm{M}_2(k)$.

The following result provides a very useful criterion for a quadratic extension of $k$ to embed into a quaternion algebra over $k$. For a proof, see Proposition 14.6.7 of \cite{V}.

\begin{proposition}\label{ABHN}
Let $k$ be a number field and $B$ be a quaternion algebra over $k$. A quadratic extension $L/k$ embeds into $B$ if and only if no prime of $k$ which ramifies in $B$ splits in $L/k$.
\end{proposition}

As an immediate application of Theorem 3.3 of \cite{CF} we have the following.

\begin{theorem}\label{containedinmaximalorder}
Let $k$ be a number field, $L$ be a quadratic field extension of $k$, and $B$ be a quaternion algebra over $k$ for which either 

\begin{enumerate}
\item[(1)] $\Ram_f(B)$ is nonempty, or 
\item[(2)] $B$ and $L/k$ do not ramify at exactly the same ( possibly empty) set of real places of $k$.
\end{enumerate}
If $u\in L$ is an integral element, then every maximal order of $B$ contains a conjugate (in $B$) of $u$.
\end{theorem}

\begin{corollary}\label{notorsioninalgebras}
Let $A$ and $B$ be quaternion division algebras over a number field $k$ with maximal orders $\mathcal O_A$ and $\mathcal O_B$. If $\mathcal O_A^\times$ contains no elements of order $n\geq 2$ and $\Ram(A)\subseteq \Ram(B)$ then $\mathcal O_B^\times$ contains no elements of order $n\geq 2$.
\end{corollary}
\begin{proof}
Suppose that $\mathcal O_B^\times$ contains an element of order $n\geq 2$.  The algebra (in $B$) generated over $k$ by this element is isomorphic to the cyclotomic field $L=k(\zeta_n)$, hence the cyclotomic extension $L=k(\zeta_n)/k$ embeds into $B$. In light of Proposition \ref{ABHN}, our hypothesis that $\Ram(A)\subseteq \Ram(B)$ implies that $L$ embeds into $A$ as well. We claim that the algebra $A$ and the extension $L/k$ satisfy one of the two conditions stated in Theorem \ref{containedinmaximalorder}. Indeed, that $A$ is a quaternion division algebra implies that $\Ram(A)$ is not empty. Therefore either $\Ram_f(A)$ is nonempty (condition (1) of Theorem \ref{containedinmaximalorder}), or there is a real place of $k$ which ramifies in $A$. If the latter is true, condition (2) of Theorem \ref{containedinmaximalorder} holds because the cyclotomic extension $L/k$ is not ramified at any real place of $k$. This proves that $A$ and $L/k$ satisfy one of the two conditions of Theorem \ref{containedinmaximalorder}. Applying the theorem to the element $u=\zeta_n$ proves that $\mathcal O_A^\times$ contains a conjugate of $\zeta_n$, and hence an element of order $n\geq 2$.
\end{proof}


\section{Arithmetic groups arising from quaternion algebras}

In this section we will review the construction of arithmetic groups arising from quaternion algebras. We refer the reader seeking a more detailed reference for this material to Maclachlan and Reid \cite{MR}.

Let $k$ be a totally real number field and $B$ be a quaternion algebra over $k$ which is unramified at a unique real place $\nu$ of $k$. We therefore have an identification $B\otimes_k k_\nu \cong \mathrm{M}_2(\mathbf R)$.  Let $\mathcal O$ be a maximal order of $B$ and $\mathcal O^1$ the multiplicative group consisting of those elements in $\mathcal O$ having reduced norm equal to one. We will denote by $\Gamma_\mathcal O^1$ the image in $\mathrm{PSL}_2(\mathbf R)$ of $\mathcal O^1$. The group $\Gamma_\mathcal O^1$ is a discrete subgroup of $\mathrm{PSL}_2(\mathbf R)$ having covolume equal to \[ \frac{8\pi \Delta_k^{3/2} \zeta_k(2)}{(4\pi^2)^{[k:\mathbf Q]}} \prod_{\nu \in\Ram_f(k)} \left(N(\nu)-1\right).\] The group $\Gamma_\mathcal O^1$ is called a {\it principal arithmetic Fuchsian group}. A Fuchsian group $\Gamma$ is {\it arithmetic} if there is a principal arithmetic Fuchsian group $\Gamma_\mathcal O^1$ which is commensurable with $\Gamma$ in the sense that the intersection $\Gamma\cap \Gamma_\mathcal O^1$ has finite index in both $\Gamma$ and $\Gamma_\mathcal O^1$. 

Let ${\bf H}^2$ denote hyperbolic $2$-space. If $M$ is a hyperbolic $2$-orbifold, then there exists a Fuchsian group $\Gamma$ such that $M={\bf H}^2/\Gamma$. We say that $M$ is {\it arithmetic} if the group $\Gamma$ is an arithmetic Fuchsian group. Similarly, we say that $M$ is principal arithmetic if the group $\Gamma$ is a principal arithmetic Fuchsian group. The orbifold $M$ is a hyperbolic $2$-manifold precisely when the group $\Gamma$ is torsion-free.

Now let $k$ be a number field having a unique complex place $\nu$. Let $B$ be a quaternion algebra over $k$ which is ramified at all real places of $k$. Then we have an identification $B\otimes_k k_\nu \cong \mathrm{M}_2(\mathbf C)$, and to any maximal order $\mathcal O$ of $B$ we can associate, in the same manner as we did above, the Kleinian group $\Gamma_\mathcal O^1\subset \mathrm{PSL}_2(\mathbf C)$. The Kleinian group $\Gamma_\mathcal O^1$ has covolume 
\[\frac{\Delta_k^{3/2}\zeta_k(2)}{(4\pi^2)^{[k:\mathbf Q]-1}} \prod_{\nu \in\Ram_f(k)} \left(N(\nu)-1\right).\] The group $\Gamma_\mathcal O^1$ is called a {\it principal arithmetic Kleinian group}, and a Kleinian group $\Gamma$ is {\it arithmetic} if there is a principal arithmetic Kleinian group which is commensurable with $\Gamma$.

Let ${\bf H}^3$ denote hyperbolic $3$-space. If $M$ is a hyperbolic $3$-orbifold, then there exists a Kleinian group $\Gamma$ such that $M={\bf H}^3/\Gamma$. We say that $M$ is {\it arithmetic} if the group $\Gamma$ is an arithmetic Kleinian group. Similarly, we say that $M$ is principal arithmetic if the group $\Gamma$ is a principal arithmetic Kleinian group.  The orbifold $M$ is a hyperbolic $3$-manifold precisely when the group $\Gamma$ is torsion-free.

Two hyperbolic $2$- or $3$-orbifolds are said to be {\it commensurable} if they share a common finite-sheeted cover (up to isometry). In the case that both orbifolds are arithmetic it is known \cite[Theorem 8.4.1]{MR} that the orbifolds are commensurable if and only if their fundamental groups are both commensurable with the same principal arithmetic group (i.e. they both arise from the same number field and the same quaternion algebra).

\begin{proposition}\label{torsionprop}
Let $\Gamma_\mathcal O^1$ be a principal arithmetic Fuchsian or Kleinian group arising from a maximal order $\mathcal O$ in a quaternion algebra $B$ defined over a number field $k$. Assume that the field $k$ is primitive; that is, that $k$ contains no proper subfields apart from $\bf Q$. If $\Ram(B)$ contains primes splitting in $k(\sqrt{-1})/k$ and $k(\sqrt{-3})/k$, then $\Gamma_\mathcal O^1$ is torsion free.
\end{proposition}
\begin{proof}
Suppose that $\Gamma_\mathcal O^1$ contains an element of finite order $m\geq 2$. Then $\mathcal O^\times$ contains a finite subgroup of order $2m$, hence there is an embedding of the cyclotomic field $k(\zeta_{2m})$ into $B$. Since $k$ is primitive, the only cyclotomic quadratic extensions of $k$ are $k(\sqrt{-1})$ and $k(\sqrt{-3})$. The result now follows from Proposition \ref{ABHN}.
\end{proof}

\subsection{Geodesics, hyperbolic elements and quadratic subfields}

Let $M$ be an arithmetic hyperbolic $2$- or $3$-orbifold with fundamental group $\Gamma$, arising from a number field $k$ and quaternion algebra $B$. The closed geodesics on $M$ are in one-to-one correspondence with the $\Gamma$-conjugacy classes of hyperbolic elements $\gamma$ of $\Gamma$. If $p_\gamma(t)$ denotes the minimal polynomial of $\gamma$ and $M(p_\gamma)$ denotes the Mahler measure of $p_\gamma$, then the length of the associated closed geodesic is equal to $\log(M(p_\gamma))$ when $M$ is a $2$-orbifold and $2\log(M(p_\gamma))$ when $M$ is a $3$-orbifold. Each closed geodesic determines a quadratic extension $k_\gamma/k$ of $k$ (i.e. the splitting field over $k$ of $p_\gamma$) which embeds into the quaternion algebra $B$.

\begin{proposition}\label{systoleramificationrelation}
Let $k$ be a number field which is either totally real or else has a unique complex place. Let $d=2$ if $k$ is totally real and $d=3$ if $k$ has a unique complex place. Let $A,B$ be quaternion algebras over $k$ which are unramified at a unique real place of $k$ if $k$ is totally real, and which are ramified at all real places of $k$ if $k$ has a unique complex place. Finally, let $\mathcal O_A, \mathcal O_B$ be maximal orders of $A,B$. 

If $\Ram_f(A)$ is nonempty and $\Ram(A) \subseteq \Ram(B)$, then $\sys\left(\mathbf H^d/\Gamma_{\mathcal O_B}^1\right)\geq \sys\left(\mathbf H^d/\Gamma_{\mathcal O_A}^1\right)$.
\end{proposition}
\begin{proof}
If $\Ram(A) \subseteq \Ram(B)$, then Proposition \ref{ABHN} implies that every quadratic extension which embeds into $B$ also embeds into $A$.  Now suppose that $\gamma$ is a hyperbolic element of $\Gamma_{\mathcal O_B}^1$ associated to a closed geodesic in $\mathbf H^d/\Gamma_{\mathcal O_B}^1$ with length $\sys\left(\mathbf H^d/\Gamma_{\mathcal O_B}^1\right)$. It follows that the extension $k_\gamma/k$ embeds into $B$. We have already that this implies that $k_\gamma/k$ also embeds into $A$. Let $u$ be a preimage in $B$ of $\gamma$ and note that $u$ is an integral element contained in $k_\gamma$. Theorem \ref{containedinmaximalorder} now implies that every maximal order of $A$ (and in particular $\mathcal O_A$) contains an element conjugate to $u$. This implies that $\mathbf H^d/\Gamma_{\mathcal O_A}^1$ contains a closed geodesic whose length is the same as the one associated to $\gamma$. The proposition follows.
\end{proof}


\section{Manifolds with the same systole}

\begin{theorem}\label{mainthm1}
Let $M$ be a hyperbolic $2$-manifold (respectively $3$-manifold) whose fundamental group is principal arithmetic. Then there exist infinitely many pairwise non-commensurable arithmetic hyperbolic $2$-manifolds (respectively $3$-manifolds) $M_1, M_2, M_3, \dots$ such that
\begin{enumerate}
\item all of the $M_i$ have the same systole as $M$, and
\item $\vol(M_{n})\leq c \cdot n^2  \cdot \vol(M_{n-1})$ for some constant $c$ depending only on $M$.
\end{enumerate}
\end{theorem}
\begin{proof}

Let $k,B$ be the invariant trace field and invariant quaternion algebra of $\pi_1(M)$, and let $\mathcal O$ be the maximal order of $B$ such that $M= \mathbf H^d/\Gamma_\mathcal O^1$. Let $L/k$ be the quadratic field extension corresponding to the systole of $M$.

Let $\p_0\in\mathbb P_k\setminus \Ram_f(B)$ and for each $i\geq 1$ let $\p_i\in\mathbb P_k\setminus \left\{\Ram_f(B)\cup\{\p_0,\dots, \p_{i-1}\}\right\}$ be a prime of $k$ which does not split in the extension $L/k$ (chosen so that $\p_i$ has minimal norm).

For each $i\geq 1$ let $B_i$ be the quaternion algebra over $k$ for which $\Ram(B_i)=\Ram(B) \cup \{\p_0,\p_i\}$. Let $\mathcal O_i$ be a maximal order of $B_i$ and define $M_i=\mathbf H^d/\Gamma_{\mathcal O_i}^1$. Corollary \ref{notorsioninalgebras} implies that $\Gamma_{\mathcal O_i}^1$ is torsion free, hence $M_i$ is a manifold.

Notice that for all $i$ we have that $\Ram(B)\subset \Ram(B_i)$. It therefore follows from Proposition \ref{systoleramificationrelation} that all of the $M_i$ have systole at least as big as that of $M$. To see that the systole of $M_i$ is in fact equal to the systole of $M$, let $\gamma\in\Gamma_\mathcal O^1$ be the hyperbolic or loxodromic element giving rise to the systole of $M$. Then $L$ is isomorphic to $k(\varphi^{-1}(\gamma))$, where $\varphi$ is the map $\varphi: B\rightarrow \isom(\mathbf H^d)$ such that $\varphi(\mathcal O^1) = \Gamma_\mathcal O^1$. A field extension $L$ of $k$ embeds into $B_i$ if and only if no prime ideal of $\Ram(B_i)$ splits in $L/k$. It thus follows from the definition of $B_i$ that $L$ embeds into $B_i$, and from Theorem \ref{containedinmaximalorder} that $\Gamma_{\mathcal O_i}^1$ contains a hyperbolic or loxodromic element with the same characteristic polynomial as $\gamma$. Because the length of a geodesic is determined by the characteristic polynomial of the associated hyperbolic or loxodromic element, it must be that \[ \sys\left( \mathbf H^d/\Gamma_{\mathcal O}^1\right) \geq \sys\left(\mathbf H^d/\Gamma_{\mathcal O_i}^1\right).\] As we've already shown the reverse inequality, this proves that $\sys(M_i)=\sys(M)$ for all $i$.

The formula for the volume of $M_i=\mathbf H^d/\Gamma_{\mathcal O_i}^1$ shows that \begin{equation}\label{volumerelation}\vol(M_i) = \vol(M) (N(\p_0)-1)(N(\p_i)-1).\end{equation} Proving property (2) of the theorem therefore reduces to the problem of finding upper bounds for a prime of least norm of $k$ which does not split in $L/k$. To obtain such an upper bound we will make use of a modification \cite[Thm 2-C]{Wang} of the bound on the least prime ideal in the Chebotarev density theorem in \cite{LMO}:

\begin{theorem}[Wang]\label{effectiveCDT}
Let $L/k$ be a finite Galois extension of number fields of degree $n$, $S$ a finite set of primes of $k$ and $[\theta]$ a conjugacy class in $\mathrm{Gal}(L/k)$. Then there is a prime ideal $\p$ of $k$ such that
\begin{enumerate}
	\item $\p$ is unramified in L and is of degree $1$ over $\mathbf Q$;
	\item $\p\notin S$;
	\item $\left(\frac{L/k}{\p}\right)=[\theta]$, and
	\item $N(\p)\leq \Delta_L^C (n\log(N_S))^2$,
\end{enumerate}
where $C$ is an absolute, effectively computable constant and $N_S=\prod_{\q\in S}N(\q)$.
\end{theorem}

Let $A=\frac{1}{( \vol(M)(N(\q_0)-1)}$. We claim that $A\leq 1$. Indeed, we clearly we have $N(\q_0) \geq 2$, so our claim will follow from the inequality $\vol(M) \geq 1$. Every closed hyperbolic $2$-manifold has area at least $4\pi$, hence the inequality $\vol(M) \geq 1$ follows from the fact that there are only two arithmetic hyperbolic $3$-manifolds with volume less than one, neither of which has a principal arithmetic fundamental group \cite{CFJR}. This proves that $A\leq 1$. We now have $N(\q_i)=1+A\vol(M_i) \leq 1+\vol(M_i)$. Applying Theorem \ref{effectiveCDT} in the case that $n=2$, $[\theta]$ represents the nontrivial element of $\mathrm{Gal}(L/k)$ and $S=\{\p_0,\dots, \p_{i-1}\}$ we obtain:

\begin{align*}
N(\p_i) &\leq 4\Delta_L^C \log(N(\p_0\cdots \p_{i-1}))^2 \\
&\leq 4\Delta_L^C \log\left(\prod_{j=0}^{i-1} 1+\vol(M_j)\right)^2 \\
&= 4\Delta_L^C \left(\sum_{j=0}^{i-1} \log\left( 1+\vol(M_j)\right)\right)^2 \\
&\leq 4\Delta_L^C i \sum_{j=0}^{i-1}\log(1+\vol(M_j))^2 \\
&\leq 4\Delta_L^C i^2 \log(1+\vol(M_{i-1}))^2 \\
&\leq 4\Delta_L^C i^2 \vol(M_{i-1}).
\end{align*}

Here we've used the Cauchy-Schwarz inequality (in the third to last inequality), as well as the inequality $\log(1+x)\leq \sqrt{x}$ (which is valid for $x\geq 0$).

Property (2) of the theorem follows from combining this bound for $N(\p_i)$ with equation (\ref{volumerelation}).
\end{proof}


\section{An upper bound on the area of an arithmetic manifold with systole bounded below}

Our proof of Theorems \ref{mainthm2} and \ref{mainthm3} will make crucial use of a very minor modification of Proposition 3.1 of \cite{LMPT2}. Although the proof is largely the same as the one appearing in \cite{LMPT2}, we include it here for the convenience of the reader.

\begin{proposition}[Proposition 3.1, \cite{LMPT2}]\label{extensionbound}
Let $k$ be a number field of degree $n$ that is totally real (respectively has a unique complex place). Let $B$ be a quaternion algebra over $k$ which is unramified at precisely one real place of $k$ (respectively is ramified at all real places of $k$). If there exists a principal arithmetic hyperbolic $2$-manifold (respectively principal arithmetic hyperbolic $3$-manifold) arising from $B$ with systole less than $x$ then there exists a quadratic field extension $L/k$ which embeds into $B$ and satisfies $N(\Delta_{L/k}) \leq  e^{2(n+x)}$.
\end{proposition}
\begin{proof}
The length of the closed geodesic associated to $\gamma$ is the logarithm of the Mahler measure of the minimal polynomial of $\gamma$ \cite[Lemma 12.3.3]{MR} or twice this quantity (depending on whether $\gamma$ is hyperbolic or loxodromic), which is equal to the height of $\gamma$ relative to $\mathbf Q(\gamma)$, hence the fact that $[\mathbf Q(\gamma):\mathbf Q]\in\{n,2n\}$ (this follows from \cite[Lemma 2.3]{CHLR} in the case of hyperbolic $3$-manifolds and from \cite[Proposition 4.13]{LMPT} in the case of hyperbolic $2$-manifolds) implies that the absolute logarithmic Weil height $h(\gamma)$ of $\gamma$ satisfies $x\geq 2nh(\gamma)$. The proposition now follows from Silverman's lower bound for the absolute logarithmic Weil height of an element in terms of the relative discriminant of the field extension it generates (see \cite{S}).
\end{proof}

\begin{theorem}\label{mainthm2}
For a given $x\in\mathbf R$ there exists a principal arithmetic hyperbolic $2$-manifold with systole at least $x$ and area at most \[4^{ce^{11x}}\] for some positive, effectively computable constant $c$. 
\end{theorem}
\begin{proof}
We may assume, without loss of generality, that $x\geq 1$. We will construct a quaternion algebra $B$ over $\bf Q$ which is unramified at the infinite place of $\bf Q$ and which has the property that no quadratic field extension $L/\bf Q$ with $\Delta_{L/\bf Q} \leq  e^{2+2x}$ embeds into $B$. If $\mathcal O$ is a maximal order of $B$ then the above proposition implies that ${\bf H}^2/\Gamma_{\mathcal O}^1$ has systole at least $x$.

Note that if $L$ is a quadratic field of discriminant $d$ then a prime $p$ not dividing $d$ splits in $L/\bf Q$ if and only if the Legendre symbol $\left(\frac{d}{p}\right) = 1$. By Linnik's theorem there exists such a prime $p$ with $p< c_1 d^{c_2}$ where $c_1$ and $c_2$ are constants that can of course be assumed to be greater than one. Because we wish to construct a quaternion algebra $B$ in which no quadratic field with discriminant less than $e^{2+2x}$ embeds, we will define $B$ to be the quaternion algebra in which every prime $p < c_1 \left( e^{2+2x}\right)^{c_2}$ ramifies. As $B$ must be unramified at the infinite place of $\bf Q$, there must be an even number of primes ramifying in $B$. Bertrand's postulate implies that there is always a prime in the interval $[x,2x]$ for all $x>1$. It follows that if there are an odd number of primes less than $c_1 \left( e^{2+2x}\right)^{c_2}$ then we may find an additional prime less than $2c_1 \left( e^{2+2x}\right)^{c_2}$. 

If $\mathcal O$ is a maximal order of $B$ then the area of the hyperbolic $2$-manifold ${\bf H}^2/\Gamma_{\mathcal O}^1$ is \[\frac{\pi}{3}\prod_{p\in\mathrm{Ram}(B)} (p-1)\leq \frac{\pi}{3} \prod_{p < 2c_1 \left( e^{2+2x}\right)^{c_2}} p <  4^{ce^{11x}} \] where $c$ is an  effectively computable constant. We note that the final inequality here follows from Erd{\H o}s' inequality $\prod_{p\leq x} p < 4^x$ (see \cite{E}) and Xylouris' proof \cite{X} that we may take $c_2< 5.5$.

To conclude our proof we need to show that ${\bf H}^2/\Gamma_{\mathcal O}^1$ is a closed hyperbolic $2$-manifold. Proposition \ref{torsionprop} implies that it suffices to show that $\Ram(B)$ contains primes which split in the quadratic extensions ${\bf Q}(\sqrt{-1})/{\bf Q}$ and ${\bf Q}(\sqrt{-3})/{\bf Q}$. The smallest prime which splits in both of these extensions is $13$. Since $B$ was defined to be the quaternion algebra in which every prime less than $c_1 \left( e^{2+2x}\right)^{c_2}$ ramifies, that $c_1,c_2> 1$ and $x\geq 1$ shows that $13\in\Ram(B)$.
\end{proof}

In the following theorem we adapt the proof of Theorem \ref{mainthm2} to the setting of hyperbolic $3$-manifolds.

\begin{theorem}\label{mainthm3}
For a given $x\in\mathbf R$ there exists a principal arithmetic hyperbolic $3$-manifold with systole at least $x$ and volume at most \[4^{ce^{2x}}\] for some positive, effectively computable constant $c$. 
\end{theorem}
\begin{proof}
We may assume, without loss of generality, that $x\geq 1$. Let $k$ be an imaginary quadratic field. Given a quadratic extension $L/k$ we will denote by $\widehat{L}$ the Galois closure of $L$ over $\bf Q$. Note that $k$ being imaginary quadratic implies that either $L=\widehat{L}$ or else $\widehat{L}$ is a Galois extension of $\bf Q$ of degree $8$ with Galois group isomorphic to the dihedral group $D_4$ (see Section 2.2 of \cite{CDO}). We will construct a quaternion algebra $B$ over $k$ which has the property that no quadratic field extension $L/k$ with $N_{k/\bf Q}(\Delta_{\widehat{L}/k}) \leq  \left(e^{2+2x}\right)^2$ embeds into $B$. This will imply, by the relative discriminant formula $\Delta_{\widehat L/k}=\Delta_{L/k}^{[\widehat L : L]} N_{L/k}(\Delta_{\widehat L/L})$,  that no quadratic field extension $L/k$ with $N_{k/\bf Q}(\Delta_{L/k}) \leq e^{2+2x}$ embeds into $B$. If $\mathcal O$ is a maximal order of $B$ then Proposition \ref{extensionbound} will imply that ${\bf H}^3/\Gamma_{\mathcal O}^1$ has systole at least $x$.

Let $L$ be a quadratic field extension of $k$. The work of Pollack \cite{P} (in the case that $L/\bf Q$ is Galois) or Ge, Milinovich and Pollack \cite[Example 1]{GMP} (in the case that $L/\bf Q$ is a $D_4$-quartic field) imply that there exists a prime $p$ splitting completely in $\widehat L/\bf Q$ with $p< c_0 \Delta_{\widehat L/\bf Q}^{1/2}$ for some constant $c_0>1$. Because we wish to construct a quaternion algebra in which no quadratic field extension $L/k$ with $N_{k/\bf Q}(\Delta_{\widehat{L}/k}) \leq  \left(e^{2+2x}\right)^2$ embeds, we recall that a prime $p$ splits completely in $L/\bf Q$ if and only if it splits completely in $\widehat L/\bf Q$ and define the algebra $B$ by \[\mathrm{Ram}(B) = \{ \p\subset \mathcal O_k : N_{k/\bf Q}(\p) \leq c_0 \Delta_{k/\bf Q}^2 e^{2+2x} \}.\] To see that $B$ has the desired property, let $L$ be a quadratic extension of $k$ which embeds into $B$ and suppose that $N_{k/\bf Q}(\Delta_{\widehat{L}/k}) \leq  \left(e^{2+2x}\right)^2$. By what we've said above, there exists a prime $p$ which splits completely in $\widehat L/\bf Q$ such that $p <  c_0 \Delta_{\widehat L/\bf Q}^{1/2}$. Let $\p$ be a prime of $k$ lying above $p$. Then $\p$ splits completely in $\widehat L/k$, and hence in $L/k$, and satisfies, by the relative discriminant formula, \[N_{k/\bf Q}(\p) = p  < c_0 \Delta_{\widehat L/\bf Q}^{1/2} = c_0 \Delta_{k/\bf Q}^{[\widehat L:k]/2}N_{k/\bf Q}(\Delta_{\widehat L/k})^{1/2} \leq c_0\Delta_{k/\bf Q}^2 e^{2+2x}.\] Therefore we have shown the existence of a prime $\p$ of $k$ which splits in $L/k$ yet ramifies in $B$. This violates Proposition \ref{ABHN}, hence $B$ has the desired property.

Let $C=\frac{|\Delta_{k/\bf Q}|^{\frac{3}{2}}\zeta_k(2)}{4\pi^2}$. If $\mathcal O$ is a maximal order of $B$ then the volume of ${\bf H}^3/\Gamma_{\mathcal O}^1$ is 
\begin{align*}
\vol({\bf H}^3/\Gamma_{\mathcal O}^1) &= C\prod_{\p \in\mathrm{Ram}(B)} (N(\p)-1) \\
&\leq C\prod_{N(\p) \leq c_0 \Delta_{k/\bf Q}^2 e^{2+2x}} N(\p) \\
&\leq C\prod_{p \leq c_0 \Delta_{k/\bf Q}^2 e^{2+2x}} p^2 \\
&\leq C \left( \prod_{p \leq c_0 \Delta_{k/\bf Q}^2 e^{2+2x}} p\right)^2 \\
&< C 4^{2c_0 \Delta_{k/\bf Q}^2 e^{2+2x}}
\end{align*}
where the last inequality follows from Erd{\H o}s' inequality $\prod_{p\leq x} p < 4^x$ (see \cite{E}).

It remains only to show that we may choose $B$ in such a way that ${\bf H}^3/\Gamma_{\mathcal O}^1$ is a hyperbolic $3$-manifold; that is, that $\Gamma_{\mathcal O}^1$ is torsion-free. Let $k$ be the imaginary quadratic field ${\bf Q}(\sqrt{-2})$. By Propsoition \ref{torsionprop} it suffices to show that $\Ram(B)$ contains primes which split in both ${\bf Q}(\sqrt{-1},\sqrt{-2})/{\bf Q}(\sqrt{-2})$ and ${\bf Q}(\sqrt{-2},\sqrt{-3})/{\bf Q}(\sqrt{-2})$. An easy calculation in SAGE shows that the prime $97$ splits completely in ${\bf Q}(\sqrt{-1},\sqrt{-2})/{\bf Q}$ and ${\bf Q}(\sqrt{-2},\sqrt{-3})/{\bf Q}$ Let $\p$ be a prime of ${\bf Q}(\sqrt{-2})$ lying above $97$. Then $\p$ splits in both ${\bf Q}(\sqrt{-1},\sqrt{-2})/{\bf Q}(\sqrt{-2})$ and ${\bf Q}(\sqrt{-2},\sqrt{-3})/{\bf Q}(\sqrt{-2})$ and has norm $N(\p)=97$. Since $|\Delta_{{\bf Q}(\sqrt{-2})/\bf Q}|=8$ and $x\geq 1$, the definition of $B$ implies that $\p\in\Ram(B)$. This shows that ${\bf H}^3/\Gamma_{\mathcal O}^1$ is a hyperbolic $3$-manifold. 

Finally, we compute (again, using SAGE) that $C 4^{2c_0 \Delta_{k/\bf Q}^2 e^{2+2x}} \leq 4^{c_010^3 e^{2x}} $. Letting $c=10^3c_0$, the theorem follows.
\end{proof}

\subsection{Can the volume bounds in Theorems \ref{mainthm2} and \ref{mainthm3} be improved upon?}\label{subsectiongrh}

While it seems likely that the volume bounds we obtained in Theorems \ref{mainthm2} and \ref{mainthm3} are far from the truth, improving them seems to be quite difficult. Consider, for example, Theorem \ref{mainthm2}. In light of Proposition \ref{extensionbound} we had to construct a quaternion algebra $B$ over $\bf Q$ in which no quadratic field with discriminant less than $e^{2+2x}$ embeds. The approach we adopted was to construct such an algebra by ensuring that for every quadratic field with discriminant less than $e^{2+2x}$ there was a prime $p\in\Ram(B)$ which split in the quadratic field. The area of a principal arithmetic hyperbolic $2$-manifold arising from such an algebra is essentially \[\prod_{p \leq c\left(e^{11x}\right)} (p-1)\] for some positive constant $c$, which led to the theorem's large area bound.

It would be natural to argue that part of the inefficiency in the aforementioned approach was to choose a {\it separate} prime splitting in each quadratic field of discriminant less than $e^{2+2x}$. For example, while the prime $5$ is the smallest prime splitting in ${\bf Q}(\sqrt{-1})/{\bf Q}$ and the prime $7$ is the smallest prime splitting in ${\bf Q}(\sqrt{-3})/{\bf Q}$, the prime $13$ splits in both extensions. It therefore would have been better to have $\Ram(B)$ contain $13$ (a contribution of $12=13-1$ to the area of the corresponding hyperbolic $2$-manifold) rather than the primes $5$ and $7$ (which together contribute $24=(5-1)(7-1)$ to the area). Taking this line of reasoning even further, perhaps it would be best to find a single prime $p$ which splits in {\it every} quadratic field of discriminant less than $e^{2+2x}$ and define $B$ by declaring $\Ram(B)=\{2,p\}$. Such a prime $p$ can be found using effective versions of the Chebotarev density theorem. In \cite{LO} it is shown that if $K/\bf Q$ is a Galois extension of number fields and $\Delta$ is the absolute value of the discriminant of $K$ then assuming the Generalized Riemann Hypothesis there exists a prime $p$ which splits completely in $K/\bf Q$ and satisfies $p<c\log^2(2\Delta)$ for some positive constant $c$. Unfortunately such an approach still yields a volume bound of the form $c_1^{c_2^x}$, as in our case we would need to define $K$ as the compositum of all quadratic fields with discriminant less than $e^{2+2x}$. The issue is that the discriminant of such a compositum is astronomical. Schmal \cite{Sch} has shown that if $a_1,\dots, a_k$ are square-free integers then the discriminant of ${\bf Q}(\sqrt{a_1},\dots, \sqrt{a_m})$ is \[\Delta = \left( 2^r \mathrm{rad}(a_1\cdots a_m)\right)^{2^{m-1}}\] where $r\in\{0,2,3\}$ and $\mathrm{rad}(n)$ denotes the radical of $n$. In our case $m$ is a constant multiple of $e^{2+2x}$, hence $\log(2\Delta)$ (as well as the area of our hyperbolic $2$-manifold) will contain a term with order of magnitude roughly $2^{e^{2+2x}}$. This is an improvement on the bound from Theorem \ref{mainthm2}, but only a minor one and is conditional on the Generalized Riemann Hypothesis.

Ultimately it seems as though the reason that the bounds in Theorems \ref{mainthm2} and \ref{mainthm3} are as big as they are is the $e^{2(n+x)}$ appearing in Proposition \ref{extensionbound}. That this term is exponential in $x$ is a consequence of Silverman's lower bound for the height of an algebraic integer in terms of the relative discriminant of the field extension it generates (see \cite{S}). While one might naively hope that Silverman's bound could be improved upon enough to result in substantially better volume bounds, the discussion of Pierce, Turnage-Butterbaugh and Wood in Section 8 of \cite{PTW} make it clear that in fact, Silverman's bound is close to the truth and likely cannot be improved upon (apart from constants).

In light of the above discussion, improving the bounds in Theorems \ref{mainthm2} and \ref{mainthm3}  seems like an interesting, though difficult, problem.


\section{A program for determining surfaces and $3$-manifolds of minimal volume with systole bounded below}

Recall that the coarea of a principal arithmetic Fuchsian group arising from a quaternion algebra $B$ over the field $\mathbf Q$ of rational numbers is given by the formula
\[\frac{\pi}{3}\prod_{p\in \Ram(B)} (p-1).\] 

In this section we will determine, for certain small values of $\ell>0$, the minimal coarea of a principal arithmetic Fuchsian group defined over $\mathbf Q$ whose associated hyperbolic $2$-manifold has systole greater than $\ell$. In order to accomplish this we wrote a program in SAGE \cite{sage} that takes as input a lower bound $\ell$ on the systole length, and returns the ramification set of a quaternion algebra over $\mathbf Q$ giving rise to a principal arithmetic Fuchsian group of smallest coarea amongst those whose associated surface has systole at least $\ell$. (We note that such a quaternion algebra will not, in general, be unique.) 

The general outline of the program is as follows.

\begin{enumerate}
\item For a given systole bound $\ell$, we compile a list of all square-free integers $d$ such that the regulator $\mathrm{Reg}_d $ of $\mathbf{Q}(\sqrt{d})$ is less than $\ell$. This is done by employing the elementary estimate (see for instance \cite[p. 2]{JLW}) \[\mathrm{Reg}_d \geq \log\left( \frac{1}{2}(\sqrt{d-4}+\sqrt{d})\right).\] If $B$ is a quaternion algebra over $\mathbf Q$ into which $\mathbf Q(\sqrt{d})$ embeds, then every principal arithmetic group arising from $B$ will contain a geodesic whose length is the regulator of $\mathbf{Q}(\sqrt{d})$, hence we must determine the ``smallest'' quaternion algebra over $\mathbf Q$ which does not admit an embedding of $\mathbf Q(\sqrt d)$ for any of the $d$ computed above.

\item Find a quaternion algebra $B'$ over $\mathbf Q$, unramified at the infinite place of $\mathbf Q$, into which none of the real quadratic fields found in the previous step embed. This is done by using the fact that $\mathbf Q(\sqrt{d})$ embeds into a quaternion algebra over $\mathbf Q$ if and only if no prime ramifying in the quaternion algebra splits in $\mathbf Q(\sqrt d)$, along with the criterion for a prime to split in a quadratic field. In practice we were able to find a quaternion algebra $B'$ by considering all subsets of the first $25$ prime numbers with cardinality $2,4,$ or $6$ as the potential ramification sets of $B'$, and then checking to see if such a quaternion algebra admitted an embedding of any of the $\mathbf Q(\sqrt d)$. If $\Gamma$ is a principal arithmetic group arising from $B'$, then the systole of $\mathbf{H}^2/\Gamma$ is greater than $\ell$. Fix one such $\Gamma$, and let $A$ be the coarea of $\Gamma$. The coarea formula shows that $A$ does not depend on the choice of maximal order used to construct $\Gamma$.

\item Let $B$ be a quaternion algebra over $\mathbf Q$, unramified at the infinite place of $\mathbf Q$, into which none of the $\mathbf Q(\sqrt{d})$ embed and whose principal arithmetic subgroups have minimal coarea amongst those of all quaternion algebras with the aforementioned property. In this step we determine an upper bound for the cardinality of $\Ram(B)$ by making use of the coarea formula inequality \[\left(\frac{\pi}{3}\prod_{p\in \Ram(B)} (p-1)\right) < A,\] where $A$ is as in (2).

\item Check all even cardinality sets of primes $S=\{p_1,\dots, p_{2k}\}$ which satisfy \begin{equation}\label{areaboundramsetiss}\frac{\pi}{3}\prod_{p\in S} (p-1) < A\end{equation} and check to see if the quaternion algebra whose ramification set is equal to $S$ admits embeddings of any of the quadratic fields $\mathbf Q(\sqrt d)$. The set $S$ for which the left hand side of (\ref{areaboundramsetiss}) is minimal is our desired ramification set. 
\end{enumerate}

The results of our computations are listed in the table below. We note that rather than listing the optimal surface areas, which are of the form \[\frac{\pi}{3}\prod_{p\in \Ram(B)} (p-1)\] and hence irrational,  we have instead listed the optimal {\it area factors} $\prod_{p\in \Ram(B)} (p-1)$, which are integral.

\vspace{1em}

\begin{center}
\begin{tabular}{| c| c |c|}
\hline
 Lower Bound for Systole Length $\ell$ & $\Ram(B)$ & Optimal Area Factor\\ 
 \hline
0.5 &  \{2,11\} & 10 \\
\hline
 1 &  \{2,31\} & 30 \\
 \hline
 1.25& \{3,31\} & 60\\
 \hline
 1.5& \{2,3,7,11\}& 120\\
 \hline
 1.75&\{3,5,7,11\}& 480\\
 \hline
 2 & \{3,5,7,11\}&480\\
 \hline
 2.25 &  \{2,3,13,41\} &960\\
 \hline
 2.5&  \{2,3,7,17\}& 2240\\
 \hline
 2.75&  \{2,3,5,7,11,13\}&5760\\
 \hline
 3&  \{2,7,29,37\}&6048\\
 \hline
 3.25 &  \{ 2,3,5,7,11,67\}&31680\\
 \hline
 3.5 &\{2,3,5,11,17,47\}&58880\\
 \hline
 3.75&\{2,3,5,19,23,27\}&114048\\
 \hline
 4 & \{2,3,19,23,31,37\}&855360\\
 \hline
 4.25& \{2,3,7,37,61,73\}&1866240\\
 \hline
 4.5&\{2,3,7,37,61,73\}&1866240\\
 \hline
 4.75&\{2,7,11,23,109,173\}&24520320\\
 \hline
 5&\{2,3,5,7,11,13,53,173\}&51517440\\
 \hline
\end{tabular}
\end{center}

\vspace{1em}

Our approach for 3-manifolds was essentially the same as our approach for surfaces, with a few notable differences:

\begin{enumerate}
    \item In the case of surfaces arising from quaternion algebras over $\mathbf{Q}$, closed geodesics are associated with quadratic fields $\mathbf{Q}(\sqrt{d})$. Such a quadratic field embeds into a quaternion algebra $B$ over $\mathbf Q$ if and only if no rational prime which ramifies in $B$ splits in $\mathbf{Q}(\sqrt{d})/\mathbf Q$. Moreover, whether a prime $p$ splits in the extension $\mathbf Q(\sqrt d)/\mathbf Q$ is easily checked; one simply needs to compute the value of the Legendre symbol $\binom{d}{p}$. For hyperbolic $3$-manifolds arising from quaternion algebras over $\mathbf Q(i)$, by contrast, one must check whether  a prime ideal in $\mathbf{Q}(i)$ splits in a quadratic extension of $\mathbf Q(i)$ (which is a field of degree $4$ over $\mathbf Q$). In practice we have found that checking this property takes significantly longer than computing a Legendre symbol.

\item In the case of surfaces, we were able to determine exactly which quadratic fields $\mathbf{Q}(\sqrt{d})$ give rise to a closed geodesic with length less than $\ell$. In the case of $3$-manifolds, we instead have to work with a bound 
$\Delta_{F/\mathbf{Q}}\leq 16e^{2(\ell+2)}$
on the discriminant $\Delta_{F/\mathbf{Q}}$ of a quadratic field extension $F$ of $\mathbf Q(i)$ that may give rise to a closed geodesic with length less than $\ell$. In light of this, there are many more degree $4$ fields to consider in the $3$-manifold case than there were quadratic fields to consider in the surface case, For example, for a systole bound of $\ell=2.7$, in the surface case we had to find quaternion algebras over $\mathbf Q$ into which $18$ quadratic fields did not embed. The number of quadratic extensions of $F$ of $\mathbf Q(i)$ that need be considered in the $3$-manifold case (for the same systole bound of $\ell=2.7$) is $1604$.

\item Because the computations for $3$-manifolds were so much more intensive than were the ones for surfaces, the $3$-manifold examples listed in the table below are no longer optimal (in the sense of having minimal volume) but simply represent examples of arithmetic hyperbolic $3$-manifolds arising from quaternion algebras over $\mathbf Q(i)$ which have systole at least $\ell$. 
\end{enumerate}

In the table below we list, for a systole bound $\ell$, a quaternion algebra $B$ over $\mathbf Q(i)$ which has the property that for any maximal order $\mathcal O$ of $B$, the hyperbolic $3$-manifold $\mathbf{H}^3/\Gamma_\mathcal O^1$ has systole at least $\ell$. We characterize $B$ by means of its ramification set $\Ram(B)$. More explicitly, we list the absolute norms of the prime ideals in $\Ram(B)$. Although in general there may be multiple prime ideals in a number field having the same absolute norm, we have found that the compactness of this presentation more than makes up for its slight ambiguity. In the final column we give the volume of the manifold $\mathbf{H}^3/\Gamma_\mathcal O^1$. (This volume does not depend on the maximal order $\mathcal O$.) 

\newpage

\begin{center}
\begin{tabular}{| c| c| c|}
\hline
\textbf{Lower Bound for Systole Length} $\ell$&$\Ram(B)$&\textbf{Volume}\\
\hline
1.0&\{2,5,9,13\}&117.24\\
\hline
1.1&\{2,5,5,9,13,13\}&5627.69\\
\hline
1.2&\{2,5,5,9,13,13\}&5627.69\\
\hline
1.3&\{2,5,5,9,13,29\}&13131.28\\
\hline
1.4&\{2,5,5,9,13,29\}&13131.28\\
\hline
1.5&\{2,5,5,13,17,61\}&56276.93\\
\hline
1.6&\{2,5,5,17, 29,53\}&78787.69\\
\hline
1.7&\{2,5,5,29, 41, 73\}&393938.4\\
\hline
1.8&\{2,9,13,17,29,53\}&682826.70\\
\hline
1.9&\{2,9,13,17,29,53\}&682826.70\\
\hline
2.0&\{2,5,5,9,17,41,41,49\}&48022976.94\\
\hline
2.1&\{2,5,5,9,17,41,41,49\}&48022976.94\\
\hline
2.2&\{2,5,9, 17,37, 41,41,49\}&$4.3221\times 10^8$\\
\hline
2.3&\{2,5,13,37,37,41,41,49\}&$1.4587\times 10^9$\\
\hline
2.4&\{2,5,13,37,37,41,41,49\}&$1.4587\times 10^9$\\
\hline
2.5&\{5, 17,17,29,37,41,41,49\}&$1.6943\times 10^{10}$\\
\hline
2.6&\{2,5,5,9,13,17,29,41,49,53\}&$2.0977\times10^{10}$\\
\hline
2.7&\{2,5,5,9,13,17,29,41,49,53\}&$2.0977\times10^{10}$\\
\hline
2.8&\{2,5,5,9,13,17,29,53,61,61\}&$3.9331\times 10^{10}$\\
\hline
2.9&\{2,5,5,9,13,17,29,53,61,61\}&$3.9331\times 10^{10}$\\
\hline
3.0&\{2,5,5,9,17,49,73,89,89,97\}&$1.6066\times 10^{12}$\\
\hline
\end{tabular}
\end{center}

\end{document}